\documentclass[12pt]{amsart}
\usepackage{latexsym,fancyhdr,amssymb,color,amsmath,amsthm,graphicx,listings,comment}
\newtheorem{thm}{Theorem}       
       \newtheorem{coro}{Corollary}

\usepackage[section]{placeins}
\let\paragraph\subsection
\title{On the cohomology of measurable sets}
\author{Oliver Knill}
\date{7/23/2023}
\address{Department of Mathematics \\ Harvard University \\ Cambridge, MA, 02138 }

\begin{document}

\begin{abstract}
If $T$ is an ergodic automorphism of a Lebesgue probability space $(X,\mathcal{A},m)$,
the set of coboundries $\mathcal{B} = dB =T(B)+B$ with symmetric difference $+$ form a subgroup 
of the set of cocycles $\mathcal{A}$. Using tools from descriptive set theory, Greg Hjorth showed in 1995 
that the first cohomology group $\mathcal{H}=\mathcal{A}/\mathcal{B}$ is uncountable. This can surprise,
given that in the case of a finite ergodic probability space, one has $\mathcal{H}=\mathbb{Z}_2$. 
Hjorth's proof used descriptive set theory in the complete metric space $(\mathcal{A},d(A,B)=m(A+B))$,
leading to the statement that $\mathcal{B}$ is meager in $\mathcal{A}$. We use a spectral genericity 
result of Barry Simon to establish the same. It leads to the statement noted first by Karl Petersen in 1973 
that for a generic $A$, the induced system $T_A$ is weakly mixing, which is slightly stronger than a 
result of Nate Friedman and Donald Ornstein about density of weakly mixing in the space of all 
induced systems $T_A$ coming from an ergodic automorphism $T$. 
\end{abstract}

\maketitle

\section{Cohomology of measurable sets}

\paragraph{}
This is a remark on a spectral and cohomological aspect in classical ergodic theory 
\cite{Friedman,Hal56,CFS,DGS,Petersen1983,Nadkarni,DeLe92,Queffelec,KH}.
A {\bf measure preserving invertible} map $T$ of a standard non-atomic {\bf Lebesgue 
probability space} $(X,\mathcal{A},m)$ defines an automorphism of the 
$\sigma$-algebra $\mathcal{A}$ which preserves measure $m(T(A))=m(A)$ for $A \in \mathcal{A}$. 
The map $T$ is called {\bf ergodic} if $0=\emptyset$ and $X$ are the only fixed points of $T$. 
The automorphism $T$ defines a group translation on the additive Abelian and Boolean 
group $(\mathcal{A},+)$. 
The set of {\bf coboundaries} $\mathcal{B}=\{ dB=B(T)+B=B(T)-B \}$ form a subgroup of $\mathcal{A}$. 
If one thinks of $\mathcal{A}$ as {\bf cocycles}, then $\mathcal{H}=\mathcal{A}/\mathcal{B}$,
the set of cocycles modulo coboundaries is the {\bf first cohomology group} of the 
dynamical system with coefficient group $\mathbb{Z}_2$. 
This coefficient group has first been considered in \cite{Kir} in the context of 
representation theory. 
Both $\mathcal{A}$ and $\mathcal{B}$ share the cardinality of the continuum so that 
the cardinality of $\mathcal{H}$ is not obvious. 
With the metric $d(A,B)=m(A + B)$, the set $\mathcal{A}$ becomes a {\bf Polish space},
a separable, metrizable and complete topological group.
The standard assumption is always that $d(A,B)=0$ is equivalent to $A=B$. Sets $A,B$
which are the same up to a set of measure zero are identified. 

\paragraph{}
We assume that the probability space is {\bf non-atomic} because if $T$ induces a 
permutation of a finite set $X$,
then $\mathcal{H}^1 = \mathcal{Z}_2^k$, where $k$ is the number of ergodic
components. Ergodicity is assumed because one can reduce the 
problem in the non-ergodic case by ergodic decomposition
\cite{Rohlin66,DGS}. The assumption to have a Lebesgue space is important but not 
really much of a restriction: 
any Borel $\sigma$-algebra of a complete separable metric space $(X,d)$ and so virtually
any probability space that appears in applications is a Lebesgue space. 
Without the Lebesgue assumption, the cohomology can be trivial, as was
first pointed out in \cite{AkCh65}, answering a question of Paul Halmos.

\paragraph{}
The construction of concrete non-coboundaries is in general not easy.
Under some circumstances,it is possible to give non-coboundaries: for example, if 
$T^2$ is ergodic, then the full space $X$ itself is not a coboundary. 
This is important for example because the {\bf set complement operation} $A \to A+X$ preserves
coboundaries if and only if $X$ is a coboundary. 

\paragraph{}
If we have ergodicity for $T,T^2$, then $X$ is not a coboundary and any $A$ which tiles
the space using transformations commuting with $T$ can not be a coboundary.
Therefore, if $X$ is a compact Abelian group with Haar measure $m$ and $Y$ is a finite 
subgroup, then $X/Y$ defines a fundamental domain $A$ which satisfies 
$ X= \{ y+A, y \in Y \}$. If $T$ is an ergodic translation on $X$
such that $T^2$ is ergodic, then $A$ can not be a coboundary. 
In the case $X=\mathbb{T}^d$ with irrational rotation for example, 
there is a countable set of explicit non-coboundaries. For $T(x)=x+\alpha$ on 
$\mathbb{T}^d$, every non-empty union of rational rectangles 
$\prod_{j=1}^d [k_j/n,(k_j+1)/n]$ can not be a coboundary. 

\paragraph{}
It is an interesting question to inquire about the structure of the group 
$\mathcal{H}$. Are there some interesting invariants which one can distill 
from $\mathcal{H}$ that allow to distinguish dynamical systems? This is not so clear.
$\mathcal{A}/\mathcal{B}$ is a topological group again it is not complete.
As it is Boolean and $A+A=0$, it only contains finite subgroups of order $\mathbb{Z}_2^k$.
We asked in \cite{Kni93diss} about the cardinality of 
$\mathcal{H}$ and whether this could depend on the dynamical system $T$. Hjorth
\footnote{Gred Hjorth (1963-2011) was an instructor at Caltech 1994-1995}
showed me the proof of the following result:

\begin{thm}
For any ergodic automorphism $T$ on a Lebesgue space, $\mathcal{H}$ is uncountable.
\end{thm} 

\paragraph{}
As a corollary, one knows in more generality that $\mathcal{H}^1(T,U)$ is uncountable, 
if $U$ is a compact topological group with normal subgroup $\mathbb{Z}_2$.
One can either take the proof and generalize it or then write an element as 
$B=\pm k$ with $k \in K=U/\mathbb{Z}_2$. If $A$ is a $\mathbb{Z}_2$ cocycle and 
$A=B(T) B^{-1}$ then $B$ can be chosen to be $\mathbb{Z}_2$-valued. It follows that
$\mathcal{H}^1(T,U)$ is uncountable. 

\paragraph{}
The proof of Hjorth already used a Baire category argument. Here is a new proof
using a result developed in the context of mathematical physics:

\begin{proof} 
Because of the Lebesgue assumption, one can assume that $X=[0,1]$ and $m=dx$ is the 
Lebesgue measure. 
We can associate to very $A \in \mathcal{A}$ of positive measure the induced
probability measure $\mu(A)=1_A/m(A)$.
Its cumulative distribution function $F_{\mu{A}}(x) = \mu(A)[0,x]$ allows to rewrite
the dynamics of the induced map $T_A$ on $[0,1]$. One has now a unitary Koopman operator $U_A: f \to f(T)$
on $L^2([0,1],dx)$ for every $A$ \cite{Koopman1931}. 
The map $A \to U_A$ is continuous if one takes the strong operator
topology on the space of unitary operators. By a result of Friedman-Ornstein
the set $\{A \; | \;  T_A $ is weakly mixing $\}$ (even mixing) 
is dense in the topological group $\mathcal{A}$. But this means it is dense in the corresponding
set $\mathcal{U} = \{ U_{T_A}, A \in \mathcal{A} \} \subset L^2([0,1],dx)$ of unitary operators, 
which is again a compact topological space in the induced metric. But
since this metric is stronger than the strong operator topology the set $\mathcal{U}$ is also 
closed in the strong operator topology. By the Simon Wonderland theorem \cite{Sim95}, 
the set of operators with no point spectrum on the orthogonal complement of constant functions
must be residual. 
\end{proof}

\paragraph{}
For the spectral theory of unitary operators, see also \cite{ChNa90}, 
about genericity, see \cite{Oxtoby}. As a comparison, here is Hjorth's proof 
(as transcribed in \cite{Knill2000}, any misunderstanding of course would be my mistake):

\begin{proof} 
$T$ defines a continuous map on the Polish group $\mathcal{A}$. 
Coboundaries are a continuous image of the automorphism
$f:A \to A(T)A^{-1}$ which has he kernel $\{-1,1\}$.
So, coboundaries are the injective image of the
automorphism $\mathcal{A}/{\rm ker}(f)$. By the
theorem of Lusin-Souslin (Theorem 15.1 in \cite{Kechris}), it is a
Borel set and so has the Baire property.
The set $\mathcal{C}=im(f)$ is not open and therefore
meager by a theorem of Pettis (Theorem 9.9 in \cite{Kechris}).
The equivalence relation on $\mathcal{A}$ given by
$A \sim B$ is $AC=BC$ is meager \cite{Der93}.
A Theorem of Mycielski (Theorem 19.1 \cite{Kechris}) implies that there are
uncountably many equivalence classes.
\end{proof}

\paragraph{}
The following corollary of the proof is a remark of Karl Petersen in 1973 \cite{Pet73}.
The density of weakly mixing has been a result of Friedman-Ornstein \cite{Friedman}. 
We can upgrade density to residual: 

\begin{coro}
The set of $A$ for which $T_A$ is weakly mixing is residual.
\end{coro}

\begin{proof}
Lemma 3.2 in \cite{Kni91} notes that $T_A^2$ is ergodic if and only if $A$ is not a coboundary.
Having no pure point spectrum on the ortho-complement of the constant functions, 
is equivalent to weak mixing.  
The set of sets $A$ for which the induced map $T_A$ is weakly mixing, is residual. 
So, for a generic measurable set $A$ of positive measure, the induced map 
$T_A$ is ergodic. 
Because weakly mixing is residual and weakly mixing implies ergodic, this especially means also
that $T_A$ is again ergodic for generic $A$ if $T$ was ergodic. 
\end{proof}

\paragraph{}
Let us also reprove also one of the first genericity results \cite{Anz51}:

\begin{coro}[Anzai]
The set of $A$ for which $T \times A$ is weakly mixing is residual.
\end{coro}

\begin{proof}
Anatoly Stepin \cite{Ste87} showed that $T \times A$ is ergodic on $X \times \mathbb{Z}_2$ 
if and only if $A$ is not a coboundary. 
\end{proof} 

\section{Remarks}

\paragraph{}
William Veech asked whether if $T$ is an irrational rotation of the circle, 
the set of $A \in \mathcal{A}$ for which $T_A$ has pure point spectrum is dense (see \cite{Pet73} p. 
229 and \cite{Con72}). If that were the case, then by the Wonderland theorem, we
would have a generic set of $A$ such that $T_A$ has {\bf purely singular continuous spectrum}.
We mentioned this Veech question in an earlier
version \cite{Kni97Preprint} to \cite{Kni97}, a document which focused on genericity alone while
\cite{Kni97} also contained quantitative mixing scales. The question of Veech can be asked for
any $T$ with pure point spectrum. We expect in this case that for a generic $A$ the system 
$T_A$ has purely singular continuous spectrum.  The interest of Veech in 1969 was in the context
of Sturmean sequences. If $A=[0,\beta]$ one gets a new dynamical system $T_A$ 
from the irrational rotation $T(x)=x+\alpha \; {\rm mod \; 1}$. 

\paragraph{}
The definition of an induced system was given by Kakutani in 1943 \cite{kakutani,Pet73}. It is a measurable
analog of a Poincar\'e return map. 
The induced system $T_A$ can upgrade ergodic to weakly mixing and so tends to make the dynamics
more chaotic. If the {\bf Kolmogorov-Sinai entropy} $h(T)$ (introduced in 1958 by Kolmogorov \cite{Rokhlin1967}
and refined by Sinai later) of $T$ is positive, then the entropy $h(T_A) = h(T)/m(A)$ of $T_A$
is larger {\bf formula of Abramov}. The case of Bernoulli shifts shows that $T_A$ can be mixing.
A theorem of Sinai assures that that every $T_A$ has a Bernoulli shift as a factor. By 
a result of Rohlin, there is then some absolutely continuous spectrum of infinite multiplicity.
Already Rohlin knew however that the set of transformations with pure absolutely continuous spectrum
is meager both in the strong and weak topology of measure preserving trasformations.

\paragraph{}
To illustrate the difficulty of cocycles and coboundaries, here is a challenge:
assume we look at the Arnold cat map $T(x,y) = (2x+y,x+y)$ on $X=\mathbb{T}^2$ which is
measure preserving and hyperbolic. Because also $T^2$ is ergodic, the entire torus 
$X$ is not a coboundary. Is $Y=[0,1/2] \times [0,1/2]$ a coboundary or not? 
More generally, which of the sets $Y_{a,b}=[0,a] \times [0,b]$ are coboundaries? 
A measurable set $Z$ such that $T(Z) + Z$ is $Y_{a,b}$ is expected to be very complicated. 

\paragraph{}
If $T$ is a permutation of a finite set $X$, a subset $A$ can serve as a jump switch to
toggle the two branches. If $T$ on $X$ has one cycle then the doubled system either has
one or two cycles. If $A$ has an even number of elements, then we have two branches. 
If $A$ is odd, we go round and reach after $n$ steps the other side meaning that we cover
all of the two branches.

\paragraph{}
Anatoly Stepin \cite{Ste87} first showed that $A$ is not a coboundary if and only if the cocycle action 
$T \times A$ is ergodic on $X \times \mathbb{Z}_2$. We used this to show that switching
stable and unstable directions of a skew flow does not necessarily kill Lyapunov exponents.
We proved in \cite{Kni91} that if we had an oracle that determines the Lyapunov exponents of 
$SL(2,\mathbb{R})$ cocycles, then we could decide whether a set is a coboundary or not. 
The problem of deciding whether $Sl(2,\mathbb{R})$ cocycle Lyapunov exponents are positive or
not is at least as hard as the coboundary problem. 

\paragraph{}
A {\it conjecture of Kirillov} asks whether two automorphisms $T,S$  with the same coboundaries 
are isomorphic. 
In the finite periodic case, Kirillov's question is negative because conjugation requires
that the orbits of the cycles have the same length.  We are not aware that the Kirillov question
has been answered. The question can even be asked in subclasses. If two irrational rotations
$x \to x+\alpha, x \to x+\beta$ have the same coboundaries. Is $\alpha=\beta$? 

\paragraph{}
Related to the cohomology problem is a more general question in group theory which goes
back to Halmos. Let $G$ be an arbitrary Abelian group and $T$ a group automorphism.
One has then the {\bf Halmos cohomology group}
$H(G,T)=G/dG$, where $dG = \{g(T)g^{-1} \; | \; g \in \mathcal{G} \}$. 
One can then define $H^k(G,T) = d^{k-1} G/d^k G$ for any $k \in \mathbb{Z}$. 
I wondered once whether $H^k(G,T)=H^{-k}(G,T)$ but has been answered negatively by 
an example of Ashbacher. 

\paragraph{}
The thesis of Jerome Depauw \cite{DeP94} looked at the second cohomology groups
\footnote{Jack Feldman informed me about de Pauw's work in 1994}. 
Such higher cohomology groups were defined independently also in \cite{Kni93diss}, as well as in 
work of Katok. The idea is simple: replace
the usual partial derivatives $\partial_j f$ of a function of several variables 
with the dynamical derivative $f(T_j)-f$. Let us only look at the two-dimensional case $d=2$, where 
$S,T$ be two commuting ergodic measure preserving transformations of the probability space. 
The standard assumption is that the 
action is {\bf free} meaning that the set of fixed points of $S^n T^m$ are of measure 
zero if $(n,m) \neq (0,0)$ Cocycles are vector fields in $\mathcal{A}$ for which the curl is zero. 
This means we have pairs $(P,Q)$ in the $\sigma$-algebra
such that the curl, the line integral along a basic plaquette in $\mathbb{Z}^2$ is zero. This means
$Q_x-P_y = P + Q(S) + P(T) + Q=0$. Coboundaries $\mathcal{B}$ are sets of the 
form $(P,Q) =[C(S) + C, C(T) + C]$. 
The curl of the gradient is zero. The first cohomology group 
are the equivalence classes of all
2-forms  modulo all 1-forms which are coboundaries. This can be generalized to cohomology groups 
$\mathcal{H}^p(T,U)$ with other coefficient groups $U=\mathbb{Z}_2$.
There is no interesting higher cohomology by a theory of Feldman and Moore \cite{FeMo77}. 
In the 2-dimensional case, this has also been treated by Depauw \cite{DeP94}.
For $d \geq 2$, the group $H^d(T,U)$ is trivial if $T$ is a free and $U$ is an Abelian Polish group.

\paragraph{}
There are no cohomological constraints in higher dimensions. 
The De Rham type cohomology is equivalent to a simplicial type group cohomology
of Eilenberg-McLane and the cohomology of ergodic equivalence relations,
as developed Feldman and Moore \cite{FeMo77}.
Douglas Lind showed in 1978 that for a general (not necessarily Abelian) 
Polish group $U$ and any measurable function $F: X \to U$, the 
equation $F(x) = P(T(x)) P(x) Q(S(x)) Q(x)$ can be solved with 
functions $P,Q$ \cite{Lin78}. The proof uses a two-dimensional version of the
{\bf Halmos-Rohlin lemma}, stating that one can approximate the 
$\mathbb{Z}^2$ action by periodic actions up to arbitrary small measure. To summarize:
for non-Abelian Polish $U$, the group $H^2(T,U)$ is trivial if $T$ is a free $\mathbb{Z}^2$ action.
This has interesting consequences in mathematical physics. 
Assume for example we have a lattice $\mathbb{Z}^2$ and
we want to generate a {\bf field} $F_{n,m}(x) =F(S^nx,T^nx)$ on the plaquette at position $(n,m)$ 
given by a two-dimensional stochastic process defined by two commuting transformations $S,T$. 
No decorrelation is assumed.  
The fields can be almost periodic for example like $F_{n,m}(x,y)= f(x+n \alpha + m \beta)$. 
Then, there is a ``lattice gauge field" $(P,Q)$ such that the curl agrees with the vector field, meaning
$F(x) = P(x) Q(Sx) P(Tx)^{-1} Q^{-1}(x)$. The curl is the 
product of the cocycle values around a plaquette. In physics, one sees $(P,Q)$ as a 
{\bf lattice gauge field}, a discrete in general non-abelian 1-form. 
The Feldman-Moore picture of ergodic equivalence relations even shows that one can realize any 
magnetic field process on a Penrose lattice as the curl of a process defined on edges of the tiling. 

\paragraph{} 
The spectral theory of unitary operators $f \to f(T)$ coming from dynamical systems were
first investigated by Koopman and von Neumann \cite{CFS}.
The history of cohomology in ergodic theory is long.
The cohomological equation $f(T)-f=g$ appears at the end of the statement 
of Hilbert's 5th problem \cite{Hil02}. It is an example of an analytic functional equation with 
only continuous but non-differentiable solution. 
Von Neumann's 1932 paper \cite{Neu} (p.641) mentions cohomology 
of dynamical systems in the context of spectra of ergodic flows. 
The term "cohomology of dynamical systems" was coined by Kirillov in \cite{Kir} in 1967.
\footnote{We owe this reference to J.P. Conze, who informed us about this around 1993.}
Other pioneers are Katok and Stepin \cite{KaSt67,KaSt70,Ste71}.
Alexandre Kirillov also already looked at the cohomology of sets $H^1(T,\mathbb{Z}_2)$.
Also the 1972 paper of \cite{Liv} uses cohomology of dynamical systems, but unaware
of \cite{Kir}. \cite{Anz51} used the name "equivalent" for  cohomologuous. 
\cite{Halmos} introduced "generalized eigenvalues".
\cite{AkCh65} continued the work aiming to find invariants for dynamical systems. 
The cohomology of dynamical systems as a special case of group cohomology
was worked out in \cite{FeMo77} in the context context of ergodic equivalence relations. 
See also Chapter 2 of \cite{Schmidt}.

\paragraph{}
The history of what happens generically and especially with respect to mixing is interesting. 
Steven Smale looked in 1967 at the general question of what happens
typically for (a not necessarily conservative) 
diffeomorphisms of a compact manifold \cite{Smale1967,Smale1971}.
Having a dense set of hyperbolic periodic orbits or the existence of homoclinic points, or the
existence of a dense set of periodic elliptic orbits is generic \cite{New75}.
In \cite{Lin78} is something about the history of mixing. 
Eberhard Hopf first defined various notions of mixing including weak mixing. 
A weakly mixing transformation that is not mixing was first constructed by 
Michiku Kakutani \cite{kakutani} and von Neumann \cite{Neu,Pet73}, but this was not published. 
Rafael Chacon \cite{Chacon1966} first showed that a change of speed of any ergodic
flow could lead to weakly mixing flows. 
The number of classes of dynamical systems with with generic weak mixing 
has exploded. A generic set of measure preserving transformations
of a probability space are weakly mixing by a theorem of Katok and Stepin \cite{KaSt67}. 
For a new proof see \cite{Kni97}.

\paragraph{}
The work \cite{Anz51} was one of the earliest works on {\bf cocycle dynamics}
and includes examples, where $T \times A$ on $X \times \mathbb{Z}_2$ has 
purely singular continuous spectrum. 
A generic volume preserving homeomorphism of a compact metric space is weakly mixing by a result of 
Oxtoby-Ulam \cite{OxUl41}, enhanced by Halmos and Rohlin pushing ergodicity to mixing. 
A generic invariant measure of a shift is weakly mixing \cite{Kni97} (enhancing \cite{Par61}).
A generic $C^1$ diffeomorphism is either weakly mixing with zero Lyapunov exponent or then 
weakly mixing and partially hyperbolic. A generic symplectomorphism has this property too by the
last theorem of Man\'e \cite{Man83} which is now a theorem \cite{AvilaCrovisierWilkinson}.
A generic measure preserving differential equation on the $2$-torus is weakly mixing \cite{HoKn98}.
A generic permutation of a Hamiltonian system with a $(k \geq 2)$-dimensional KAM torus 
produces a weakly mixing invariant torus \cite{Kni99}.
A generic polygonal billiard is weakly mixing \cite{ChaikaForni} and
 \cite{GutkinKatok}. Remarkable is \cite{SabogalTroubetzkoy} showing that 
a Baire generic vertical-horizontal billiard is weakly mixing along a set of directions of full measure.
In \cite{Knill2000}, we asked whether a generic convex body in a square is 
weakly mixing. This is indeed the case: for strictly convex Sinai scatter bodies, 
one has no pure point spectrum by Pesin theory \cite{KatokStrelcyn}, for rational polygonal 
scatterers one has pure point spectrum. Rational polygonal bodies are dense in all convex bodies. 
One believes that a generic right angle triangle produces a weakly mixing,
non-mixing dynamics. Unproven is whether in the smooth case, there is a generic set of diffeomorphisms which 
have a weakly mixing invariant component of positive measure. The vague attractor of Kolmogrov (VAK) 
picture suggests that there should be ergodic mixing components of positive measure mixed with 
quasi-periodic motion. For a more recent list of problems \cite{Pesin2007}.

\bibliographystyle{plain}

\begin{thebibliography}{10}

\bibitem{AkCh65}
M.A. Akcoglu and R.V. Chacon.
\newblock Generalized eigenvalues of automorphisms.
\newblock {\em Proc. Amer. Math. Soc.}, 16:676--680, 1965.

\bibitem{Liv}
A.Livsic.
\newblock Cohomology of dynamical systems.
\newblock {\em Izv. Akad. Nauk SSSR}, 36:1278--1301, 1972.

\bibitem{Anz51}
H.~Anzai.
\newblock Ergodic skew product transformations on the torus.
\newblock {\em Osaka Math. J.}, 3:83--99, 1951.

\bibitem{AvilaCrovisierWilkinson}
A.~Avila, S.~Crovisier, and A.~Wilkinson.
\newblock Diffeomorphisms with positive metric entropy.
\newblock {\em Pub. Math. IHES}, 124:319--347, 2016.

\bibitem{Chacon1966}
R.~V. Chacon.
\newblock Change of velocity in flows.
\newblock {\em J. Math. Mech.}, 15:417--431, 1966.

\bibitem{ChaikaForni}
J.~Chaika and G.~Forni.
\newblock Weakly mixing polygonal billiards.
\newblock https://arxiv.org/abs/2003.00890, 2020.

\bibitem{ChNa90}
J.R. Choksi and M.G.Nadkarni.
\newblock Baire category in spaces of measures, unitary operators and
  transformations.
\newblock In {\em Invariant Subspaces and Allied Topics}, pages 147--163.
  Narosa Publ. Co., New Delhi, 1990.

\bibitem{Con72}
J.P. Conze.
\newblock Equations functionnelles et syst\`ems induits en th\'eorie ergodique.
\newblock {\em Z. Wahr. Verw. Geb.}, 23:75--82, 1972.

\bibitem{CFS}
I.P. Cornfeld, S.V.Fomin, and Ya.G.Sinai.
\newblock {\em Ergodic Theory}, volume 115 of {\em {Grundlehren} der
  mathematischen {Wissenschaften} in {Einzeldarstellungen}}.
\newblock Springer Verlag, 1982.

\bibitem{DeLe92}
A.~DelJunco and M.~Lemancyk.
\newblock Generic spectral properties of measure-preserving maps and
  applications.
\newblock {\em Proc. Amer. Math. Soc.}, 115:725--736, 1992.

\bibitem{DGS}
M.~Denker, C.~Grillenberger, and K.~Sigmund.
\newblock {\em Ergodic Theory on Compact Spaces}.
\newblock Lecture Notes in Mathematics 527. Springer, 1976.

\bibitem{DeP94}
J.~DePauw.
\newblock {Th\'eor\`emes ergodiques pour cocycle de degr\'e 2}.
\newblock {Th\`ese de doctorat, Universit\'e de Bretagne Occidentale}, 1994.

\bibitem{Der93}
J.-M. Derrien.
\newblock Crit\`eres d'ergodicit\'e de cocycles en escalier. {Exemples}.
\newblock {\em C.\ R.\ Acad.\ Sc.\ Paris}, 316:73--76, 1993.

\bibitem{FeMo77}
J.~Feldman and C.Moore.
\newblock Ergodic equivalence relations, cohomology and von {N}eumann algebras
  {I,II}.
\newblock {\em Transactions of the AMS}, 234:289--359, 1977.

\bibitem{Friedman}
N.A. Friedman.
\newblock {\em Introduction to Ergodic Theory}.
\newblock {Van Nostrand-Reinhold, Princeton, New York}, 1970.

\bibitem{GutkinKatok}
E.~Gutkin and A.~Katok.
\newblock Weakly mixing billiards.
\newblock In {\em Holomorphic Dynamics}, volume 1345 of {\em Lecture Notes in
  Mathematics}, 1988.

\bibitem{Halmos}
P.~Halmos.
\newblock {\em Lectures on ergodic theory}.
\newblock The mathematical society of {Japan}, 1956.

\bibitem{Hal56}
P.R. Halmos.
\newblock {\em Lectures on Ergodic Theory}.
\newblock The Mathematical Society of Japan, 1956.

\bibitem{Hil02}
D.~Hilbert.
\newblock Mathematische probleme.
\newblock In {\em Gesammelte Abhandlungen {III}}. Berlin 1935, 1902.

\bibitem{HoKn98}
A.~Hof and O.~Knill.
\newblock Zero dimensional singular continuous spectrum for smooth differential
  equations on the torus.
\newblock {\em Ergodic Theory and Dynamical Systems}, 18:879--888, 1998.

\bibitem{KH}
A.~Katok and B.~Hasselblatt.
\newblock {\em Introduction to the modern theory of dynamical systems},
  volume~54 of {\em Encyclopedia of Mathematics and its applications}.
\newblock Cambridge {University} Press, 1995.

\bibitem{KatokStrelcyn}
A.~Katok and J.-M. Strelcyn.
\newblock {\em Invariant manifolds, entropy and billiards, smooth maps with
  singularities}, volume 1222 of {\em Lecture notes in mathematics}.
\newblock Springer-Verlag, 1986.

\bibitem{KaSt67}
A.B. Katok and A.M. Stepin.
\newblock Approximations in ergodic theory.
\newblock {\em Russ. Math. Surveys}, 22:77--102, 1967.

\bibitem{KaSt70}
A.B. Katok and A.M. Stepin.
\newblock Metric properties of measure preserving homemorphisms.
\newblock {\em Russ. Math. Surveys}, 25:191--220, 1970.

\bibitem{Kechris}
A.S. Kechris.
\newblock {\em Classical Descriptive Set Theory}, volume 156 of {\em Graduate
  Texts in Mathematics}.
\newblock Springer-Verlag, Berlin, 1994.

\bibitem{Kir}
A.A. Kirillov.
\newblock Dynamical systems, factors and representations of groups.
\newblock {\em Russian Math. Surveys}, 22:63--75, 1967.

\bibitem{Kni97Preprint}
O.~Knill.
\newblock Singular continuous spectrum in ergodic theory.
\newblock Caltech, March 20, 1995.

\bibitem{Kni91}
O.~Knill.
\newblock {The upper {L}yapunov exponent of ${\rm {S}{L}}(2,{\bf {R}})$
  cocycles: discontinuity and the problem of positivity}.
\newblock In {\em Lyapunov exponents ({Oberwolfach}, 1990)}, volume 1486 of
  {\em Lecture Notes in Math.}, pages 86--97. Springer, Berlin, 1991.

\bibitem{Kni93diss}
O.~Knill.
\newblock Spectral, ergodic and cohomological problems in dynamical systems.
\newblock {PhD Theses, ETH Z\"urich}, 1993.

\bibitem{Knill2000}
O.~Knill.
\newblock On the cohomology of a discrete abelian group action.
\newblock Preprint Caltech, 1996, last edit 2000.

\bibitem{Kni97}
O.~Knill.
\newblock Singular continuous spectrum and quantitative rates of weakly mixing.
\newblock {\em Discrete and continuous dynamical systems}, 4:33--42, 1998.

\bibitem{Kni99}
O.~Knill.
\newblock Weakly mixing invariant tori of {H}amiltonian systems.
\newblock {\em Commun. Math. Phys.}, 204:85--88, 1999.

\bibitem{Koopman1931}
B.O. Koopman.
\newblock Hamiltonian systems and transformations in hilbert space.
\newblock {\em Proc Natl Acad Sci USA}, 17:315--318, 1931.

\bibitem{Lin78}
D.A. Lind.
\newblock Products of coboundaries for commuting nonsingular automorphisms.
\newblock {\em Z. Wahrsch. verw. Gebiete}, 43:135--139, 1978.

\bibitem{Queffelec}
{M. Queff\'elec}.
\newblock {\em Substitution Dynamical Systems---Spectral Analysis}, volume 1294
  of {\em Lecture Notes in Mathematics}.
\newblock Springer, 1987.

\bibitem{Nadkarni}
M.G. Nadkarni.
\newblock {\em Spectral Theory of Dynamical Systems}.
\newblock {Birkh\"auser}, 1998.

\bibitem{New75}
S.E. Newhouse.
\newblock Quasi-elliptic periodic points in conservative dynamical systems.
\newblock {\em Amer. J. Math.}, 99:1061--1087, 1975.

\bibitem{Oxtoby}
J.C. Oxtoby.
\newblock {\em Measure and Category}.
\newblock Springer Verlag, New York, 1971.

\bibitem{OxUl41}
J.C. Oxtoby and S.M. Ulam.
\newblock Measure-preserving homeomorphisms and metrical transitivity.
\newblock {\em Annals of Mathematics}, 42:874--920, 1941.

\bibitem{Par61}
K.R. Parthasarathy.
\newblock On the categorie of ergodic measures.
\newblock {\em Ill. J. Math.}, 5:648--656, 1961.

\bibitem{Pesin2007}
Y.~Pesin.
\newblock Existence and genericity problems for dynamical systems with nonzero
  lyapunov exponents.
\newblock {\em Regular and Chaotic Dynamics}, 12:476--489, 2007.

\bibitem{Pet73}
K.~Petersen.
\newblock Spectra of induced transformations.
\newblock In {\em Recent Advances in topological dynamics}. Springer Verlag,
  1973.
\newblock Lecture Notes in Mathematics, 318.

\bibitem{Petersen1983}
K.~Petersen.
\newblock {\em Ergodic theory}.
\newblock Cambridge {University} Press, Cambridge, 1983.

\bibitem{Man83}
{R. Man\'e}.
\newblock Oseledec's theorem from a generic viewpoint.
\newblock In {\em Proceedings of the International Congress of Mathematicians},
  1983.
\newblock Proceedings of the International Congress of Mathematicians,
  Warschau, 1983.

\bibitem{Rohlin66}
V.A. Rohlin.
\newblock Selected topics from the metric theory of dynamical systems.
\newblock {\em Amer. Math. Soc. Transl. Ser 2}, 49:171--240, 1966.

\bibitem{Rokhlin1967}
V.A. Rokhlin.
\newblock Lectures on the entropy theory of measure-preserving transformations.
\newblock {\em Russian Mathematical Surveys}, 22:1--52, 1967.

\bibitem{SabogalTroubetzkoy}
A.M. Sabogal and S.~Troubetzkoy.
\newblock Weakly mixing polygonal billiards.
\newblock {\em Bull. London Math. Soc.}, 49:141--147, 2017.

\bibitem{Schmidt}
K.~Schmidt.
\newblock {\em Algebraic ideas in ergodic theory}, volume~76 of {\em Regional
  conference Series in mathematics}.
\newblock AMS, Providence, 1989.

\bibitem{Sim95}
B.~Simon.
\newblock Operators with singular continuous spectrum: {I}. {General}
  operators.
\newblock {\em Annals of Mathematics}, 141:131--145, 1995.

\bibitem{kakutani}
S.Kakutani.
\newblock A generalization of {B}rouwer's fixed point theorem.
\newblock {\em Duke Math. J.}, 8:457--459, 1941.

\bibitem{Smale1971}
S.~Smale.
\newblock Stability and genericity in dynamical systems.
\newblock {\em Seminaire N. Bourbaki}, 374:177--185, 1871.

\bibitem{Smale1967}
S.~Smale.
\newblock Differentiable dynamical systems.
\newblock {\em Bull. AMS}, 73:757--817, 1967.

\bibitem{Ste71}
A.M. Stepin.
\newblock Cohomologies of automorphism groups of a {Lebesgue} space.
\newblock {\em Functional Analysis and Applications}, 5:167--168, 1971.

\bibitem{Ste87}
A.N. Stepin.
\newblock Spectral properties of generic dynamical systems.
\newblock {\em Math. USSR Izvestiya}, 29:159--192, 1987.

\bibitem{Neu}
J.~von Neumann.
\newblock Zur {Operatorenmethode in der Mechanik}.
\newblock {\em Ann. Math.}, 33:587--642, 1932.

\end{thebibliography}

\end{document}